\documentclass[12pt,twoside]{amsart}
\usepackage{amsmath}
\usepackage{amsthm}
\usepackage{amsfonts}
\usepackage{amssymb}
\usepackage{latexsym}
\usepackage[all]{xy}
\usepackage{epsfig}
\usepackage{amscd,mathrsfs,graphicx,mathrsfs}
\usepackage[numbers,sort&compress]{natbib}

\date{}
\allowdisplaybreaks[4] \footskip=15pt

\renewcommand{\uppercasenonmath}[1]{}

\numberwithin{equation}{section} \theoremstyle{plain}
\newtheorem*{thm*}{Theorem A}
\newtheorem*{thm**}{Theorem B}
\newtheorem{thm}{Theorem}[section]
\newtheorem{cor}[thm]{Corollary}
\newtheorem*{cor*}{Corollary}
\newtheorem{lem}[thm]{Lemma}
\newtheorem*{lem*}{Lemma}

\newtheorem*{prop*}{Proposition}

\newtheorem*{rem*}{Remark}
\newtheorem{exa}[thm]{Example}
\newtheorem*{exa*}{Example}
\newtheorem*{obs*}{Observation}
\newtheorem{df}[thm]{Definition}
\newtheorem*{df*}{Definition}

\newtheorem*{conj*}{Conjecture}
\newtheorem*{ack*}{ACKNOWLEDGEMENTS}

\newcommand{\pf}{\noindent\begin {proof}}
\newcommand{\epf}{\end{proof}}

\begin{document}
\begin{center}
{\Large \bf Gorenstein projective modules and Frobenius extensions
 \footnotetext{
E-mail address: renwei@fudan.edu.cn.}}

\vspace{0.5cm}     Ren Wei\\
{\small School of Mathematical Sciences, Chongqing Normal University, Chongqing 401331, China\\
School of Mathematical Sciences, Fudan University, Shanghai 200433, China}
\end{center}


\bigskip
\centerline { \bf  Abstract}
\leftskip10truemm \rightskip10truemm
We prove that for a Frobenius extension, if a module over the extension ring is Gorenstein projective, then its underlying module over the the base ring is Gorenstein projective; the converse holds if the Frobenius extension is either left-Gorenstein or separable (e.g. the integral group ring extension $\mathbb{Z}\subset \mathbb{Z}G$).

Moreover, for the Frobenius extension $R\subset A=R[x]/(x^2)$, we show that: a graded $A$-module is Gorenstein projective in $\mathrm{GrMod}(A)$, if and only if its ungraded $A$-module is Gorenstein projective, if and only if its underlying $R$-module is Gorenstein projective. It immediately follows that an $R$-complex is Gorenstein projective if and only if all its items are Gorenstein projective $R$-modules.
\bigskip


\leftskip0truemm \rightskip0truemm \vbox to 0.2cm{}

\section { \bf Introduction}
A module $M$ is said to be Gorenstein projective \cite{EJ00} if there exists a totally acyclic complex of projective modules $\mathbf{P}:=\cdots \rightarrow P_{1}\rightarrow P_{0}\rightarrow P_{-1}\rightarrow \cdots$ such that $M = \mathrm{Ker}(P_{0}\rightarrow P_{-1})$. The study of Gorenstein projective modules plays an important role in some areas such as representation theory of Artin algebras, the theory of stable and singularity categories, and cohomology theory of commutative rings. Especially, for finitely generated Gorenstein projective modules, there are several different terminologies in the literature, such as modules of G-dimension zero, maximal Cohen-Macaulay modules and totally reflexive modules.

For a given ring $R$, it is important to find a ``well-behaved'' extension ring $A$ in the sense that some useful information can transfer between $R$ and $A$. In this paper, we intend to study relations of Gorenstein projective modules along Frobenius extensions of rings.
The theory of Frobenius extensions was developed by Kasch \cite{Kas54} as a generalization of Frobenius algebras, and was further studied by Nakayama-Tsuzuku \cite{NT60} and Morita \cite{Mor67}. A classical example of Frobenius extension is the integral group ring extension $\mathbb{Z}\subset\mathbb{Z}G$ for a finite group $G$. Other examples include Hopf subalgebras \cite{Sch92}, finite extensions of enveloping algebras of Lie super-algebras \cite{BF93}, enveloping algebras of Lie coloralgebras \cite{FMS97}. We refer to a lecture due to Kadison \cite{Kad99}.

We are partly inspired by an observation of Buchweitz \cite[Section 8.2]{Buc87}: for a finite group $G$, a $\mathbb{Z}G$-module, or equivalently an integral representation of $G$, is maximal Cohen-Macaulay over $\mathbb{Z}G$  if and only if the underlying $\mathbb{Z}$-module is maximal Cohen-Macaulay,  or equivalently, the underlying $\mathbb{Z}$-module is free. In \cite{Chen13}, Chen introduces a generalization of Frobenius extension, called the totally reflexive extension of rings, and proves that totally reflexive modules transfer along such extension. However, is this true for not necessarily finitely generated Gorenstein projective modules? As it is pointed out at the end of \cite{Chen13}, a different argument is needed.

The first main result gives a partial answer to the above question; see Theorems \ref{thm 2.1} and \ref{thm 2.3}.
\begin{thm*}\label{thm A} Let $R\subset A$ be a Frobenius extension, $M$ a left $A$-module. If $M$ is Gorenstein projective in $\mathrm{Mod}(A)$, then the underlying $R$-module $M$ is Gorenstein projective; the converse holds if $R\subset A$ is either a left-Gorenstein or a separable Frobenius extension.
\end{thm*}

We remark that $\mathbb{Z}\subset \mathbb{Z}G$ is both a left-Gorenstein and a separable Frobenius extension, so Buchweitz's observation is true for not necessarily finitely generated Gorenstein projective modules. In order to prove Theorem \ref{thm 2.1}, we need a fact that over a left-Gorenstein ring, $(\mathcal{GP}, \mathcal{W})$ is a cotorsion pair \cite{BR07}. We use $\mathcal{GP}$ to denote the class of Gorenstein projective modules, and $\mathcal{W}$ to denote the class of modules with finite projective dimension. However, we further show in Theorem \ref{thm 2.2} that the cotorsion pair $(\mathcal{GP}, \mathcal{W})$ is cogenerated by a set. This result generalizes \cite[Theorem 8.3]{Hov02} from Iwanaga-Gorenstein rings to left-Gorenstein rings. It seems to be of particular interest, since this will induce a cofibrantly generated model structure on the category of modules by applying Hovey's correspondence \cite[Theorem 2.2]{Hov02}, such that the associated homotopy category is exactly the stable category $\underline{\mathcal{GP}}$.

The second inspirational example of this paper is the ring extension $R\subset A=R[x]/(x^{2})$. One can also view $A$ as a graded ring with a copy of $R$ (generated by 1) in degree 0 and a copy of $R$ (generated by $x$) in degree 1. It is shown in Theorem \ref{thm 3.1} that:
\begin{thm**}\label{thm B} A graded $A$-module is Gorenstein projective in $\mathrm{GrMod}(A)$, if and only if its ungraded module is Gorenstein projective in $\mathrm{Mod}(A)$, if and only if its underlying module is Gorenstein projective in $\mathrm{Mod}(R)$.
\end{thm**}

For the graded ring $A=R[x]/(x^{2})$, there is an observation that the category $\mathrm{GrMod}(A)$ is automatically isomorphic to the category $\mathrm{Ch}(R)$ of $R$-complexes; see for example \cite{GH10}. So a Gorenstein projective graded $A$-module is precisely the Gorenstein projective $R$-complex introduced by Enochs and Garc\'{i}a Rozas \cite{EGR98}. It is immediate that (Corollary \ref{cor 3.1}): an $R$-complex is Gorenstein projective if and only if all its items are Gorenstein projective $R$-modules; see also \cite[Theorem 1]{Yang11}. This generalizes \cite[Theorem 4.5]{EGR98} and \cite[Theorem 3.1]{LZ11} by removing the conditions that the base ring $R$ is Iwanaga-Gorenstein and is right coherent and left perfect, respectively.

The paper is organized as follows. In Section 2, we introduce the notion of left-Gorenstein Frobenius extensions, and it is shown that over left-Gorenstein rings, $(\mathcal{GP}, \mathcal{W})$ is a cotorsion pair cogenerated by a set. We study the separable Frobenius extensions. Then, Theorem A is proved. In Section 3, we focus on Gorenstein projective graded $R[x]/(x^{2})$-modules, and we prove the result in Theorem B.

\section{\bf Gorenstein projective modules over Frobenius extensions}
Throughout, all rings are associative with a unit. Homomorphisms of rings are required to send the unit to the unit. Let $R$ be a ring.
A left $R$-module $M$ is sometimes written as $_{R}M$. For two left $R$-modules $M$ and $N$, denote by $\mathrm{Hom}_{R}(M, N)$ the abelian
group consisting of left $R$-homomorphisms between them. A right $R$-module $M$ is sometimes written as $M_{R}$. We identify right $R$-modules
with left $R^{op}$-modules, where $R^{op}$ is the opposite ring of $R$. For two right $R$-modules $M$ and $N$, the abelian group of right $R$-homomorphisms is denoted by $\mathrm{Hom}_{R^{op}}(M, N)$. We denote by $\mathrm{Mod}(R)$ the category of left $R$-modules, and
$\mathrm{Mod}(R^{op})$ the category of right $R$-modules. Let $S$ be another ring. An $R$-$S$-bimodule $M$ is written as $_{R}M_{S}$.

We always denote a ring extension $\iota: R\hookrightarrow A$ by $R\subset A$. The natural bimodule $_{R}A_{R}$ is given by
$rar^{'}:= \iota(r)\cdot a\cdot \iota(r^{'})$. Similarly, we consider $_{R}A$ and $_{R}A_{A}$ etc. For a ring extension $R\subset A$, there is a restricted functor $Res: \mathrm{Mod}(A)\rightarrow \mathrm{Mod}(R)$ sends $_{A}M$ to $_{R}M$, given by $rm:=\iota(r)m$. The structure map $\iota$ is usually suppressed. In the opposite direction, there are functors $T = A\otimes_{R}-: \mathrm{Mod}(R)\rightarrow \mathrm{Mod}(A)$ and
$H= \mathrm{Hom}_{R}(A, -): \mathrm{Mod}(R)\rightarrow \mathrm{Mod}(A)$. It is clear that $(T, Res)$ and $(Res, H)$ are adjoint pairs.

\medskip
\noindent{\bf 2.1 Frobenius extensions}
\smallskip

We refer to \cite[Definition 1.1, Theorem 1.2]{Kad99} for the definition of Frobenius extensions.

\begin{df}\label{def 1} A ring extension $R\subset A$ is a Frobenius extension, provided that one of the following equivalent conditions holds:\\
\indent $(1)$ The functors $T = A\otimes_{R}-$ and $H= \mathrm{Hom}_{R}(A, -)$ are naturally equivalent.\\
\indent $(2)$ $_{R}A$ is finite generated projective and $_{A}A_{R}\cong (_{R}A_{A})^{*}= \mathrm{Hom}_{R}(_{R}A_{A}, R)$.\\
\indent $(3)$ $A_{R}$ is finite generated projective and $_{R}A_{A}\cong (_{A}A_{R})^{*}= \mathrm{Hom}_{R^{op}}(_{A}A_{R}, R)$.\\
\indent $(4)$ There exists an $R$-$R$-homomorphism $\tau: A\rightarrow R$ and elements $x_i$, $y_i$ in $A$, such that for any $a\in A$, one has
$\sum\limits_{i}x_i\tau(y_ia) = a$ and $\sum\limits_{i}\tau(ax_i)y_i = a$.
\end{df}

\begin{lem}\label{lem 2.1} Let $R\subset A$ be a Frobenius extension of rings, $M$ a left $A$-module. If $_{A}M$ is Gorenstein projective, then the underlying left $R$-module $_{R}M$ is also Gorenstein projective.
\end{lem}

\begin{proof}
Let $M$ be a Gorenstein projective left $A$-module. There exists a totally acyclic complex, i.e. an acyclic complex of projective $A$-modules
$\mathbf{P}:=\cdots \rightarrow P_{1}\rightarrow P_{0}\rightarrow P_{-1}\rightarrow \cdots$ with $\mathrm{Hom}_{A}(\mathbf{P}, P)$ being an
acyclic complex for each projective $A$-module $P$, such that $M = \mathrm{Ker}(P_{0}\rightarrow P_{-1})$. Note that each $P_i$ is a projective
left $R$-module. Then by restricting $\mathbf{P}$ one gets an acyclic complex of projective $R$-modules.

Let $Q$ be a projective left $R$-module. It follows from isomorphisms $\mathrm{Hom}_{R}(A, Q)\cong A\otimes_{R}Q$ that $\mathrm{Hom}_{R}(A, Q)$ is a projective left $A$-modules. Then the complex $\mathrm{Hom}_{A}(\mathbf{P}, \mathrm{Hom}_{R}(A, Q))$ is acyclic. Moreover, there are isomorphisms $\mathrm{Hom}_{R}(\mathbf{P}, Q)\cong\mathrm{Hom}_{R}(A\otimes_{A}\mathbf{P}, Q)\cong \mathrm{Hom}_{A}(\mathbf{P}, \mathrm{Hom}_{R}(A, Q))$. This implies that the complex $\mathrm{Hom}_{R}(\mathbf{P}, Q)$ is acyclic, and hence the underlying $R$-module $M$ is Gorenstein projective.
\end{proof}

\begin{lem}\label{lem 2.2} Let $R\subset A$ be a Frobenius extension of rings, $M$ a left $A$-module. If the underlying module $_{R}M$ is Gorenstein projective, then the following hold:\\
\indent $(1)$ For any projective $A$-module $P$ and any $i>0$, $\mathrm{Ext}^{i}_{A}(M, P)=0$.\\
\indent $(2)$ $A\otimes_{R}M$ is a Gorenstein projective left $A$-module.
\end{lem}

\begin{proof}
(1) For any left $A$-module $M$ and any left $R$-module $N$, there are isomorphisms
$$\mathrm{Hom}_{A}(M, A\otimes_{R}N)\cong \mathrm{Hom}_{A}(M, \mathrm{Hom}_{R}(A, N))\cong\mathrm{Hom}_{R}(A\otimes_{A}M, N)\cong \mathrm{Hom}_{R}(M, N).$$
Moreover, by replacing $_{A}M$ with an $A$-projective resolution $\mathbb{P}^{\bullet}$ of $M$ and observing that $\mathbb{P}^{\bullet}$ is also an $R$-projective resolution of $_{R}M$, we have an isomorphism of cohomology $\mathrm{Ext}_{A}^{i}(M, A\otimes_{R}N)\cong \mathrm{Ext}_{R}^{i}(M, N)$ for any $i>0$.

Let $P$ be a projective left $A$-module. There is a split epimorphism $\theta: A\otimes_{R}P\rightarrow P$ of $A$-modules given by $\theta(a\otimes_{R}x)=ax$ for any $a\in A$ and $x\in P$, and then $P$ is a direct summand of $A\otimes_{R}P$. Since $P$ is projective as a left $R$-module, and $_{R}M$ is Gorenstein projective by assumption, we have $\mathrm{Ext}_{A}^{i}(M, A\otimes_{R}P)\cong \mathrm{Ext}_{R}^{i}(M, P)=0$, and then $\mathrm{Ext}^{i}_{A}(M, P)=0$ as desired.

(2) Let $\mathbf{P}:=\cdots \rightarrow P_{1}\rightarrow P_{0}\rightarrow P_{-1}\rightarrow \cdots$ be a totally acyclic complex of projective $R$-modules such that $_{R}M = \mathrm{Ker}(P_{0}\rightarrow P_{-1})$. It is easy to see that $A\otimes_{R}\mathbf{P}$ is an acyclic complex of projective $A$-modules, and $A\otimes_{R}M = \mathrm{Ker}(A\otimes_{R}P_{0}\rightarrow A\otimes_{R}P_{-1})$. Moreover, for any projective $A$-module $P$, the complex $\mathrm{Hom}_{A}(A\otimes_{R}\mathbf{P}, P)\cong\mathrm{Hom}_{R}(\mathbf{P}, P)$ is acyclic. So $A\otimes_{R}M$
is a Gorenstein projective left $A$-module.
\end{proof}

\medskip
\noindent{\bf 2.2~~~~ Left-Gorenstein Frobenius extensions}
\medskip

Following \cite[Theorem VII2.5]{BR07}, a ring $\Lambda$ is called left-Gorenstein provided the category $\mathrm{Mod}(\Lambda)$ of left $\Lambda$-modules is a Gorenstein category. This is equivalent to the condition that the global Gorenstein projective dimension of $\Lambda$ is finite. By \cite[Theorem 10.2.14]{EJ00}, each Iwanaga-Gorenstein ring (i.e. two-sided noetherian ring with left and right self-injective dimension) is left-Gorenstein. The converse is not true in general. For example, let $S_{n} = S[x_{1}, x_{2},\cdots, x_{n}]$ be the polynomial ring in $n$ indeterminates over a non-noetherian hereditary ring $S$. Let $R_{i} = S_{i-1} \otimes S_{i-1}$ be the trivial extension of $S_{i-1}$ by $S_{i-1}$
for $i \geq 1$ (set $S_{0} = S$). Then $R_{i}$ is a left-Gorenstein ring for every $i \geq 1$, whereas $R_{i}$ is non-noetherian, and hence is not Iwanaga-Gorenstein.

\begin{df}\label{def 2} Let $R\subset A$ be a Frobenius extension. Then $R\subset A$ is called a left-Gorenstein Frobenius extension  provided in addition that $A$ is left-Gorenstein.
\end{df}

\begin{thm}\label{thm 2.1} Let $R\subset A$ be a left-Gorenstein Frobenius extension of rings, $M$ a left $A$-module. Then $M$ is a Gorenstein projective left $A$-module if and only if the underlying left $R$-module $M$ is Gorenstein projective.
\end{thm}

\begin{proof}
By Lemma \ref{lem 2.1}, it suffices to prove that when the underlying module $_{R}M$ is Gorenstein projective, $M$ is a Gorenstein projective left $A$-module.

Note that over a left-Gorenstein ring $A$, a module $M$ is Gorenstein projective if and only if $\mathrm{Ext}^{i}_{A}(M, N)=0$ for any module $N$ of finite projective dimension and any $i>0$; see \cite{BR07} or Theorem \ref{thm 2.2} below. Assume that $N$ is an $A$-module with projective dimension $n$. Then there is an exact sequence $0\rightarrow K\rightarrow P\rightarrow N\rightarrow 0$ of $A$-modules, where $P$ is projective and $K$ is of projective dimension $n-1$. By induction on the projective dimension of modules, it is deduced from Lemma \ref{lem 2.2}(1) that $\mathrm{Ext}^{i}_{A}(M, N)\cong \mathrm{Ext}^{i+1}_{A}(M, K)=0$. The assertion follows.
\end{proof}

For a finite group $G$, it is easy to see that the integral group ring $\mathbb{Z}G$ is Iwanaga-Gorenstein, since there is an exact sequence
$0\rightarrow \mathbb{Z}G\rightarrow \mathbb{Q}G\rightarrow \mathbb{Q}/\mathbb{Z}G\rightarrow 0$ of left or right $\mathbb{Z}G$-modules, where
$\mathbb{Q}G = \mathrm{Hom}_{\mathbb{Z}}(\mathbb{Z}G, \mathbb{Q})$ is an injective $\mathbb{Z}G$-module, and similarly $\mathbb{Q}/\mathbb{Z}G$ is injective.

\begin{cor}\label{cor 2.1} Let $G$ be a finite group, $M$ a left $\mathbb{Z}G$-module. Then $M$ is a Gorenstein projective left $\mathbb{Z}G$-module if and only if the underlying left $\mathbb{Z}$-module $M$ is Gorenstein projective.
\end{cor}

Recall that a pair of classes $(\mathcal{X}, \mathcal{Y})$ of modules is a cotorsion pair provided that $\mathcal{X} =  {^\perp}\mathcal{Y}$ and $\mathcal{Y} = \mathcal{X}^{\perp}$, where $^{\perp}\mathcal{Y} = \{X \mid \mathrm{Ext}^{1}_{}(X, Y) = 0,~~\forall~~Y\in \mathcal{Y}\}$ and $\mathcal{X}^{\perp} = \{Y \mid \mathrm{Ext}^{1}_{}(X, Y) = 0,~~\forall~~X\in \mathcal{X}\}$. The cotorsion pair $(\mathcal{X}, \mathcal{Y})$ is said to be cogenerated by a set $\mathcal{S}$ if $\mathcal{S}{^{\perp}} = \mathcal{Y}$. Over an Iwanaga-Gorenstein ring $A$, it follows from \cite[Theorem 8.3]{Hov02} that $(\mathcal{GP}, \mathcal{W})$ is a cotorsion pair cogenerated by a set, where $\mathcal{GP}$ is the class of Gorenstein projective modules, and $\mathcal{W}$ is the class of modules with finite projective dimension.

It follows from \cite{BR07} that over a left-Gorenstein ring, $(\mathcal{GP}, \mathcal{W})$ is a cotorsion pair. We have more in the next result, which also generalizes \cite[Theorem 8.3]{Hov02} from Iwanaga-Gorenstein rings to left-Gorenstein rings. It seems to be of particular interest, since by Hovey's correspondence \cite[Theorem 2.2]{Hov02} between cotorsion pairs and model structures, we get a cofibrantly generated Gorenstein projective model structure on the category of modules. Moreover, the homotopy category associated with the model structure is exactly the stable category $\underline{\mathcal{GP}}$.

\begin{thm}\label{thm 2.2}
Let $A$ be a left-Gorenstein ring. The cotorsion pair $(\mathcal{GP}, \mathcal{W})$ is cogenerated by a set.
\end{thm}

\begin{proof}
Note that over a left-Gorenstein ring, a module is Gorenstein projective if and only if it is a syzygy of an acyclic complex of projectives.
We denote by $\mathrm{ac}\widetilde{\mathcal{P}}(A)$ the class of all acyclic complexes of projective $A$-modules. For a module $M$, we use $\underline{M}$ to denote the complex with $M$ concentrated in degree zero. The cardinal of a complex $\mathbf{C}:= \cdots \rightarrow C_{i+1}\rightarrow C_{i}\rightarrow C_{i-1}\rightarrow \cdots$ is defined to be $|\mathbf{C}| = |\bigoplus_{i\in \mathbb{Z}}C_{i}|$.

\medskip
\noindent{\bf Claim 1.} Let $\aleph > |A|+\aleph_{0}$ be an infinite cardinal, $\mathbf{P}:= \cdots \rightarrow P_{1}\stackrel{\partial_{1}}\rightarrow P_{0}\stackrel{\partial_{0}}\rightarrow P_{-1}\stackrel{\partial_{-1}}\rightarrow \cdots$ be
a complex in $\mathrm{ac}\widetilde{\mathcal{P}}(A)$. Let $\mathbf{C} = \underline{M}$ be a subcomplex of $\mathbf{P}$, where $M\leq P_{0}$ is a submodule with $|M|\leq \aleph$. There exists a subcomplex $\mathbf{D}\in \mathrm{ac}\widetilde{\mathcal{P}}(A)$, such that $|\mathbf{D}|\leq \aleph$, $\mathbf{C}\leq \mathbf{D}$ and $\mathbf{D}/\mathbf{C}\in \mathrm{ac}\widetilde{\mathcal{P}}(A)$.

It follows from the Kaplansky theorem that every projective module is a direct sum of countably generated projective modules.
Then $P_{n}=\bigoplus_{i\in I_{n}}P_{n,i}$ with each $P_{n,i}$ countably generated.
Let $S^{1}_{0}=\bigoplus_{i\in J_{0}}P_{0,i}$, where $J_{0}=\{i\in I_{0}| M\cap P_{0,i}\neq 0\}$. Then $M\leq S^{1}_{0}$,
$|S^{1}_{0}|\leq\aleph$, $S^{1}_{0}$ and $P_{0}/S^{1}_{0}$ are projective modules. We can now consider the acyclic complex
$$\xymatrix@C=15pt{
  \cdots   \ar[r]^{}& L^{1}_{4}\ar[r]^{\partial_{4}}& L^{1}_{3}\ar[r]^{\partial_{3}} & L^{1}_{2} \ar[r]^{\partial_{2}} & L^{1}_{1} \ar[r]^{\partial_{1}} & S^{1}_{0} \ar[r]^{\partial_{0}\quad} &
  \partial_{0}(S^{1}_{0})
  \ar[r]^{} & 0}, \eqno{(S1)}$$
where $L^{1}_{i}$ is a submodule of $P_{i}$ of cardinality less than or equal to $\aleph$ such that
$\partial_{i}(L^{1}_{i}) =\mathrm{Ker}(\partial_{i-1}|_{L^{1}_{i-1}})$ for all $i > 0$ (we let $L^{1}_{0}=S^{1}_{0}$).
Now, we can embed $\partial_{0}(S^{1}_{0})$ into a projective submodule $S^{2}_{-1}\leq P_{-1}$, such that $|S^{2}_{-1}|\leq\aleph$
and $P_{-1}/S^{2}_{-1}$ being a projective module. Then consider the acyclic complex
$$\xymatrix@C=15pt{
  \cdots   \ar[r]^{}& L^{2}_{3}\ar[r]^{\partial_{3}}& L^{2}_{2}\ar[r]^{\partial_{2}} & L^{2}_{1} \ar[r]^{\partial_{1}} & L^{2}_{0} \ar[r]^{\partial_{0}} & S^{2}_{-1} \ar[r]^{\partial_{-1}\quad} &
  \partial_{-1}(S^{2}_{-1})
  \ar[r]^{} & 0}, \eqno{(S2)}$$
where each $L^{2}_{i}$ is taken as before. If we embed $L^{2}_{0}$ into a projective submodule $S^{3}_{0}$ of $P_{0}$ and construct $L^{3}_{i}$ as before, we then get a complex which is also acyclic:
$$\xymatrix@C=15pt{
  \cdots   \ar[r]^{}& L^{3}_{2}\ar[r]^{\partial_{2}} & L^{3}_{1} \ar[r]^{\partial_{1}} & S^{3}_{0} \ar[r]^{\partial_{0}\quad\quad}
  & S^{2}_{-1}+\partial_{0}(S^{3}_{0}) \ar[r]^{\quad\partial_{-1}} &\partial_{-1}(S^{2}_{-1})\ar[r]^{} & 0}. \eqno{(S3)}$$
Now choose a projective submodule $S^{4}_{1}\leq P_{1}$ with $|S^{4}_{1}|\leq\aleph$, which contains $L^{3}_{1}$, such that $P_{1}/S^{4}_{1}$ is a projective module. We then get an acyclic complex
$$\xymatrix@C=15pt{
  \cdots   \ar[r]^{}& L^{4}_{3}\ar[r]^{\partial_{3}}& L^{4}_{2}\ar[r]^{\partial_{2}} & S^{4}_{1} \ar[r]^{\partial_{1}\quad\quad} & S^{3}_{0}+\partial_{1}(S^{4}_{1}) \ar[r]^{\partial_{0}}
  & S^{2}_{-1}+\partial_{0}(S^{3}_{0}) \ar[r]^{\quad\partial_{-1}} &\partial_{-1}(S^{2}_{-1})\ar[r]^{} & 0}. \eqno{(S4)}$$
Now we turn over and get the following acyclic complexes
$$\xymatrix@C=15pt{
  \cdots   \ar[r]^{}& L^{5}_{3}\ar[r]^{\partial_{3}}& L^{5}_{2}\ar[r]^{\partial_{2}} & L^{5}_{1} \ar[r]^{\partial_{1}} & S^{5}_{0} \ar[r]^{\partial_{0}\quad\quad}
  & S^{2}_{-1}+\partial_{0}(S^{5}_{0}) \ar[r]^{\quad\partial_{-1}} &\partial_{-1}(S^{2}_{-1})\ar[r]^{} & 0}, \eqno{(S5)}$$
$$\xymatrix@C=15pt{
  \cdots   \ar[r]^{}& L^{6}_{3}\ar[r]^{\partial_{3}}& L^{6}_{2}\ar[r]^{\partial_{2}} & L^{6}_{1} \ar[r]^{\partial_{1}} & L^{6}_{0}\ar[r]^{\partial_{0}}
  & S^{6}_{-1}\ar[r]^{\partial_{-1}\quad} &\partial_{-1}(S^{6}_{-1})\ar[r]^{} & 0}, \eqno{(S6)}$$
$$\xymatrix@C=15pt{
  \cdots   \ar[r]^{}& L^{7}_{2}\ar[r]^{\partial_{2}}& L^{7}_{1}\ar[r]^{\partial_{1}} & L^{7}_{0} \ar[r]^{\partial_{0}} & L^{7}_{-1}\ar[r]^{\partial_{-1}}
  & S^{7}_{-2}\ar[r]^{\partial_{-2}\quad} &\partial_{-2}(S^{7}_{-2})\ar[r]^{} & 0}, \eqno{(S7)}$$
$$\xymatrix@C=15pt{
  \cdots   \ar[r]^{}& L^{8}_{2}\ar[r]^{\partial_{2}}& L^{8}_{1}\ar[r]^{\partial_{1}} & L^{8}_{0} \ar[r]^{\partial_{0}} &S^{8}_{-1}\ar[r]^{\partial_{-1}\quad\quad}
  &S^{7}_{-2}+\partial_{-1}(S^{8}_{-1})\ar[r]^{\quad\partial_{-2}} &\partial_{-2}(S^{7}_{-2})\ar[r]^{} & 0}, \eqno{(S8)}$$
$$\xymatrix@C=15pt{
  \cdots   \ar[r]^{}& L^{9}_{2}\ar[r]^{\partial_{2}}& L^{9}_{1}\ar[r]^{\partial_{1}} & S^{9}_{0} \ar[r]^{\partial_{0}\quad\quad} & S^{8}_{-1}+\partial_{0}(S^{9}_{0})\ar[r]^{\partial_{-1}\quad}
  & S^{7}_{-2}+\partial_{-1}(S^{8}_{-1})\ar[r]^{\quad\partial_{-2}} &\partial_{-2}(S^{7}_{-2})\ar[r]^{} & 0}, \eqno{(S9)}$$
where $S^{k}_{i}$ are projective submodules of $P_{i}$, such that $|S^{k}_{i}|\leq\aleph$ and $P_{i}/S^{k}_{i}$ being projective.

If we continue this zig-zag procedure, we then find acyclic complexes $(Sn)$ for all $n$, in such a way that there are infinitely many $n$
with $(Sn)_{i}$ a projective submodule of $P_{i}$ for each $i\in \mathbb{Z}$. Furthermore, we have $M\leq (Sn)_{0}$ and
$|(Sn)|\leq \aleph_{0}\cdot\aleph\leq \aleph$  for any $n$. Let $\mathbf{D}$ be the direct limit of $(Sn)$, $n\in \mathbb{Z}$. Then $\mathbf{D}$ is the desired acyclic complex of projective modules.

\medskip
\noindent{\bf Claim 2.} Let $\aleph > |A|+\aleph_{0}$ be an infinite cardinal, and $M$ a Gorenstein projective $A$-module. Then for any submodule $K\leq M$ with $|K|\leq \aleph$, there exists a submodule $N$ of $M$, such that $K\leq N$, $N$ and $M/N$ are Gorenstein projective modules, and $|N|\leq \aleph$.

There exists an acyclic complex $\mathbf{P}:=\cdots \rightarrow P_{1}\rightarrow P_{0}\rightarrow P_{-1}\rightarrow \cdots$
of projective $A$-modules, such that $M = \mathrm{Ker}(P_{0}\rightarrow P_{-1})$. By the above argument, for complex $\mathbf{C}= \underline{K}$, there is an acyclic subcomplex $\mathbf{D}:=\cdots \rightarrow D_{1}\rightarrow D_{0}\rightarrow D_{-1}\rightarrow \cdots$ of projective $A$-modules, such that $|\mathbf{D}|\leq \aleph$, $\mathbf{C}\leq \mathbf{D}$ and $\mathbf{D}/\mathbf{C}\in \mathrm{ac}\widetilde{\mathcal{P}}(A)$.
Thus, $N= \mathrm{Ker}(D_{0}\rightarrow D_{-1})$ is the desired submodule of $M$.

\medskip
\noindent{\bf Claim 3.}  $(\mathcal{GP}, \mathcal{W})$ is a cotorsion pair cogenerated by a set.

Let $M\in \mathcal{GP}$. By transfinite induction we can find a continuous chain of submodules of $M$, say $\{M_{\alpha}; \alpha< \lambda\}$, for some ordinal number $\lambda$ such that $M=\cup_{\alpha< \lambda}M_{\alpha}$; $M_{0}$, $M_{\alpha+1}/M_{\alpha}$ are in $\mathcal{GP}$, and $|M_{0}|\leq \aleph$, $|M_{\alpha+1}/M_{\alpha}|\leq \aleph$ for any $\alpha< \lambda$. But since $\mathcal{GP}$ is closed under extensions and direct limits, in fact each $M_{\alpha}$ belongs to $\mathcal{GP}$, and so every module in $\mathcal{GP}$ is the direct union of a continuous chain of submodules in  $\mathcal{GP}$ with cardinality less than or equal to $\aleph$. Note that $\mathcal{GP}$ is a Kaplansky class (see \cite{ELR02, Gil07}), or equivalently, a deconstructible class (see \cite{Sto13}).

Thus, if we let $\mathcal{S}$ be a representative set of modules $M\in \mathcal{GP}$ with $|M|\leq \aleph$, then a module $N\in\mathcal{GP}^{\perp}$
if and only if $\mathrm{Ext}_{A}^{1}(M, N) =0$ for any $M\in \mathcal{S}$, that is, $(\mathcal{GP}, \mathcal{GP}^{\perp})$ is cogenerated by the set $\mathcal{S}$ (see e.g. \cite[Theorem 7.3.4]{EJ00}). The equality $\mathcal{GP}^{\perp}=\mathcal{W}$ follows by a standard argument, so we omit it. This completes the proof.
\end{proof}

\medskip
\noindent{\bf 2.3~~~~ Separable Frobenius extensions}
\medskip

The separable algebra enjoys some of the attractive properties of semisimple algebras. The separability of rings and algebras has been concerned by many authors, for example, Azumaya, Auslander and Goldman. We refer to \cite[Charpter 10]{Pie82} and \cite[Section 2.4]{Kad99} for separable rings (algebras).

\begin{df}\label{def 3}
A ring extension $R\subset A$ is separable provided the multiplication map $\varphi: A\otimes_{R}A\rightarrow A$ ($a\otimes_{R}b\rightarrow ab$) is a split epimorphism of $A$-bimodules. If $R\subset A$ is simultaneously a Frobenius extension and a separable extension, then it is called a separable Frobenius extension.
\end{df}

Note that for any left $A$-module $M$, there is a natural map $\theta: A\otimes_{R}M\rightarrow M$ given by $\theta(a\otimes_{R}m)=am$ for any $a\in A$ and $m\in M$. It is easy to check that $\theta$ is surjective, and as an $R$-homomorphism it is split. However, in general $\theta$ is not split as an $A$-homomorphism. The following is analogous to the results in \cite{Pie82} for separable algebras over commutative rings.

\begin{lem}\label{lem 2.3} The following are equivalent:\\
\indent $(1)$ $R\subset A$ is a separable extension.\\
\indent $(2)$ For any $A$-bimodule $M$, $\theta: A\otimes_{R}M\rightarrow M$ is a split epimorphism of $A$-bimodules.\\
\indent $(3)$ There exists an element $e\in A\otimes_{R}A$, such that $\varphi(e)=1_{A}$ and $ae=ea$ for any $a\in A$.
\end{lem}

\begin{proof}
(1) is a special case of (2) by letting $M=A$. Now assume (1) holds. For an $A$-bimodule $M$, we have the following diagram
$$\xymatrix@C=40pt{
  (A\otimes_{R}A)\otimes_{A}M\ar[r]^{\quad\quad\varphi\otimes \mathrm{id}_{M}}\ar[d]_{\mu} & A\otimes_{A} M \ar[d]^{\pi}\\
  A\otimes_{R}M \ar[r]^{\quad\theta} & M
  }$$
where $\pi$ is a natural isomorphism, and $\mu$ is the composition
$$(A\otimes_{R}A)\otimes_{A}M\longrightarrow A\otimes_{R}(A\otimes_{A}M)\stackrel{\mathrm{id}_{A}\otimes \pi}\longrightarrow A\otimes_{R}M.$$
An easy calculation shows that the diagram commutes. Let $\psi: A\rightarrow A\otimes_{R}A$ be a homomorphism of $A$-bimodules such that $\varphi\psi = \mathrm{id}_{A}$. If we define $\chi = \mu(\psi\otimes \mathrm{id}_{M})\pi^{-1}$, then $\chi$ is an $A$-bimodule homomorphism such that $\theta\chi = \mathrm{id}_{M}$. Hence, the epimorphism of $A$-bimodules $\theta: A\otimes_{R}M\rightarrow M$ is split.

It remains to prove the equivalence of (1) and (3). If $\varphi: A\otimes_{R}A \rightarrow A$ is split, then $e=\psi(1_{A})\in A\otimes_{R}A$, such that $\varphi(e) = \varphi(\psi(1_{A}))= 1_{A}$, and $ae=\psi(a1_{A})=\psi(1_{A}a)=ea$ for any $a\in A$. Conversely, if there is an element $e\in A\otimes_{R}A$ satisfying (3), and $\psi: A\rightarrow A\otimes_{R}A$ is defined by $\psi(a)=ae$, then $\varphi\psi(a) = \varphi(ae) = a\varphi(e)=a$. Moreover, $\psi(ab)=(ab)e=a(be)=a\psi(b)$,
and $\psi(ab)= a(be)=a(eb)= (ae)b=\psi(a)b$, that is, $\psi$ is an $A$-bimodule homomorphism. Thus, $R\subset A$ is separable.
\end{proof}

\begin{exa}\label{exa 2.1}
\indent $(1)$ For a finite group $G$, $\mathbb{Z}\subset \mathbb{Z}G$ is a separable Frobenius extension. Indeed, let $e=\frac{1}{|G|}\sum_{g\in G}g\otimes_{\mathbb{Z}}g^{-1}\in \mathbb{Z}G\otimes_{\mathbb{Z}}\mathbb{Z}G$, where $|G|$ is the order of $G$. It is direct to check that $e$ satisfies the condition (3) of the above lemma.\\
\indent $(2)$ (\cite[Example 2.7]{Kad99}) Let $F$ be a field and set $A=M_{4}(F)$. Let $R$ be the subalgebra of $A$ with $F$-basis consisting of the idempotents and matrix units: $$e_{1}=e_{11}+e_{44}, e_{2}=e_{22}+e_{33}, e_{21}, e_{31}, e_{41}, e_{42}, e_{43}.$$
Then $R\subset A$ is a separable Frobenius extension.
\end{exa}

If $R\subset A$ is a separable extension, it follows from the above argument that as left $A$-modules, $M$ is a direct summand of $A\otimes_{R}M$. The following is immediate from Lemma \ref{lem 2.1} and Lemma \ref{lem 2.2}(2).

\begin{thm}\label{thm 2.3} Let $R\subset A$ be a separable Frobenius extension, $M$ a left $A$-module. Then
$M$ is a Gorenstein projective $A$-module if and only if the underlying $R$-module $M$ is Gorenstein projective.
\end{thm}

We note that relationship between Gorenstein projective modules over ring extensions are considered in other conditions, for example, in \cite{HS12} for excellent extensions of rings, and in \cite{LS13} for cross product of Hopf algebras.

\section {\bf Gorenstein projective graded $R[x]/(x^{2})$-modules}

Throughout this section, $R$ is an arbitrary ring, $A=R[x]/(x^{2})$ is the quotient of the polynomial ring, where $x$ is a variable
which is supposed to commute with all the elements of $R$.

\begin{lem}\label{lem 3.1}
The extension of rings $R\subset A$ is a Frobenius extension.
\end{lem}

\begin{proof}
It is clear that $A_{R}$ is a finitely generated projective module. There is an $R$-$A$-homomorphism $\varphi: A\rightarrow \mathrm{Hom}_{R^{op}}(_{A}A_{R}, R)$ given by $\varphi(r_0+r_1x)(s_0+s_1x)=r_{0}s_{0}+r_{0}s_{1}+r_{1}s_{0}$
for any $r_0 + r_1x$ and $s_0 + s_1x$ in $A$, and a homorphism $\psi: \mathrm{Hom}_{R^{op}}(_{A}A_{R}, R)\rightarrow A$ which maps any $f\in\mathrm{Hom}_{R^{op}}(_{A}A_{R}, R)$ to an element $f(x)+(f(1)-f(x))x$ in $A$. It is direct to check that $\varphi\psi= \mathrm{id}$ and $\psi\varphi=\mathrm{id}$. The assertion follows.
\end{proof}

One can view $A$ as a graded ring with a copy of $R$ (generated by 1) in degree 0
and a copy of $R$ (generated by $x$) in degree 1, and 0 otherwise. A graded $A$-module $M$ is an $A$-module with a additive subgroup decomposition $M=\bigoplus_{i\in \mathbb{Z}}M^{i}$, such that $A^{i}M^{j}\subset M^{i+j}$ for all $i$ and $j$. Consider graded $A$-modules $M$ and $N$. An $A$-linear map $f:M\rightarrow N$ has degree $d$ if $f(M^i)\subset N^{i+d}$. The set of all degree $d$ maps from $M$ to $N$ is denoted by $\mathrm{Hom}_{A}(M,N)_{d}$. We define $\mathrm{Hom}_{\mathrm{Gr}}(M,N):=\mathrm{Hom}_{A}(M,N)_{0}$. The category $\mathrm{GrMod}(A)$ consists of graded left $A$-modules and the morphisms are taken to be the graded morphism of degree zero. Note that by forgetting the grading on a module, there is naturally a functor $\mathrm{GrMod}(A)\rightarrow \mathrm{Mod}(A)$.

There is an observation that the category $\mathrm{GrMod}(A)$ is isomorphic to the category $\mathrm{Ch}(R)$ of $R$-complexes, where $M=\bigoplus_{i\in \mathbb{Z}}M^{i}$ corresponds to the cochain complex $\cdots\rightarrow M^{i-1}\rightarrow M^{i}\rightarrow M^{i+1}\rightarrow\cdots$ of $R$-modules, with the differential corresponding to multiplication by $x$; see for example \cite{GH10}.
It is clear that the isomorphism of categories between $\mathrm{GrMod}(A)$ and $\mathrm{Ch}(R)$ automatically preserves projectives.

Let $\mathcal{C}$ be an abelian category with enough projectives. An object $M\in \mathcal{C}$ is said to be Gorenstein projective if it is a syzygy of a totally acyclic complex of projectives. The notion of Gorenstein projective complexes is introduced by Enochs and Garc\'{i}a Rozas \cite[Definition 4.1]{EGR98} as Gorenstein projective objects in $\mathrm{Ch}(R)$. We call the Gorenstein projective objects in $\mathrm{GrMod}(A)$ to be Gorenstein projective graded $A$-modules.

\begin{obs*}\label{obs 3.1}
Let $M = \bigoplus_{i\in \mathbb{Z}}M^{i}\in\mathrm{GrMod}(A)$. Then $M$ is a Gorenstein projective graded $A$-module if and only if $\cdots\rightarrow M^{i-1}\rightarrow M^{i}\rightarrow M^{i+1}\rightarrow\cdots$ is a Gorenstein projective $R$-complex.
\end{obs*}

The main result of this section is stated as follows.

\begin{thm}\label{thm 3.1} Let $M\in \mathrm{GrMod}(A)$ be a graded $A$-module. The following are equivalent:\\
\indent $(1)$ $M$ is Gorenstein projective in $\mathrm{GrMod}(A)$.\\
\indent $(2)$ $M$ is Gorenstein projective in $\mathrm{Mod}(A)$.\\
\indent $(3)$ $M$ is Gorenstein projective in $\mathrm{Mod}(R)$.
\end{thm}

The next result is immediate, which generalizes \cite[Theorem 4.5]{EGR98} by removing the prerequisite that the base ring is Iwanaga-Gorenstein, and generalizes \cite[Theorem 3.1]{LZ11} by removing the condition that the base ring is right coherent and left perfect; see also \cite[Theorem 1]{Yang11}.

\begin{cor}\label{cor 3.1}
Let $M$ be an $R$-complex. Then $M$ is Gorenstein projective in $\mathrm{Ch}(R)$ if and only if each item $M^{i}$ is Gorenstein projective
in $\mathrm{Mod}(R)$.
\end{cor}

There is a result due to Gillespie and Hovey \cite[Proposition 3.8]{GH10}: every dg-projective complex over $R$ is a Gorenstein projective $A$-module, and the converse holds if $R$ is left and right noetherian and of finite global dimension. It is well-known that the projective dimension of a Gorenstein projective module is either zero or infinity, see for example \cite[Proposition 10.2.3]{EJ00}. If $R$ is a ring of finite global dimension, then dg-projective $R$-complex and Gorenstein projective $R$-complex coincide. So the assumption of noetherian ring in
\cite[Proposition 3.8]{GH10} is not needed.

In the rest of this section, we are devoted to prove Theorem \ref{thm 3.1}. For any graded $A$-module $M$ and $d\in \mathbb{Z}$, we define $M[d]$ to be a shift of $M$, which is equal to $M$ as an ungraded $A$-module but has grading $M[d]^{i}=M^{i+d}$. For any $R$-module $N$, we denote by $\overline{N}$ the graded $A$-module with $N$ in degree -1 and 0; the differential corresponding to multiplication by $x$ is exactly the identity of $N$. The next result is well-known.

\begin{lem}\label{lem 3.2}
Let $N$ be a graded $A$-module. Then $N$ is projective in $\mathrm{GrMod}(A)$ if and only if $N$ is projective in $\mathrm{Mod}(A)$. If we consider $N$ as an $R$-complex, then $N$ is projective in $\mathrm{Ch}(R)$, and there is a family of projective $R$-modules $\{P^i\}_{i\in \mathbb{Z}}$ such that $N = \prod_{i\in \mathbb{Z}}\overline{P^{i}}[-i]$.
\end{lem}

\begin{lem}\label{lem 3.3}
Let $M$ be a graded $A$-module. If $M$ is Gorenstein projective in $\mathrm{GrMod}(A)$, then the ungraded module $M$ is Gorenstein projective in $\mathrm{Mod}(A)$.
\end{lem}

\begin{proof}
Let $M\in \mathrm{GrMod}(A)$. Assume that there is a totally acyclic complexes of projectives $\mathbb{P}:=\cdots \rightarrow P_{1}\rightarrow P_{0}\rightarrow P_{-1}\rightarrow \cdots$ in $\mathrm{GrMod}(A)$, such that $M = \mathrm{Ker}(P_{0}\rightarrow P_{-1})$. Note that every item $P_{j} = \bigoplus_{i\in \mathbb{Z}}P_{j}^{i}$ is a projective module in $\mathrm{Mod}(A)$, and then $\mathbb{P}$ is also an exact sequence of projective modules in $\mathrm{Mod}(A)$.

Let $D$ be a projective left $R$-module. Then $\overline{D}[-i]$ is projective in $\mathrm{GrMod}(A)$ for any $i\in \mathbb{Z}$. Note that for any $N\in\mathrm{GrMod}(A)$, we have $\mathrm{Hom}_{\mathrm{Gr}}(N, \overline{D}[-i])\cong \mathrm{Hom}_{\mathrm{Ch}(R)}(N, \overline{D}[-i])\cong \mathrm{Hom}_{R}(N^{i}, D)$. Then, the complex $\mathrm{Hom}_{\mathrm{Gr}}(\mathbb{P}, \overline{D}[-i])\cong \mathrm{Hom}_{R}(\mathbb{P}^{i}, D)$ is acyclic, where $\mathbb{P}^{i}:= \cdots \rightarrow P_{1}^{i}\rightarrow P_{0}^{i}\rightarrow P_{-1}^{i}\rightarrow \cdots$. Moreover, the complex $\mathrm{Hom}_{R}(\mathbb{P}, D)$ is acyclic for any projective $R$-module $D$.

Let $Q$ be a projective left $A$-module. Then $Q$ is a projective left $R$-module, and $A\otimes_{R}Q$ is a projective $A$-module. The canonical epimorphism $\theta: A\otimes_{R}Q\rightarrow Q$ of $A$-modules is split. Moreover, by the argument in Lemma \ref{lem 2.2}, there is an isomorphism
$\mathrm{Hom}_{A}(\mathbb{P}, A\otimes_{R}Q)\cong \mathrm{Hom}_{R}(\mathbb{P}, Q)$. This implies that the complex $\mathrm{Hom}_{A}(\mathbb{P}, A\otimes_{R}Q)$ is acyclic. Hence, $\mathrm{Hom}_{A}(\mathbb{P}, Q)$ is acyclic. It yields that $\mathbb{P}$ is a totally acyclic complex of projective $A$-modules, and $M$ is Gorenstein projective in $\mathrm{Mod}(A)$.
\end{proof}

\begin{lem}\label{lem 3.4}
Let $M\in\mathrm{GrMod}(A)$. If $M$ is Gorenstein projective in $\mathrm{Mod}(A)$, then there is an exact sequence $0\rightarrow M\rightarrow N\rightarrow L\rightarrow 0$ in $\mathrm{GrMod}(A)$ with $N$ projective, $L$ Gorenstein projective in $\mathrm{Mod}(A)$; and moreover, it also remains exact after applying $\mathrm{Hom}_{\mathrm{Gr}}(-, P)$ for any projective module $P\in\mathrm{GrMod}(A)$.
\end{lem}

\begin{proof}
We consider the graded $A$-module $M=\bigoplus_{i\in \mathbb{Z}}M^{i}$ as an $R$-complex with differential $\delta$ of degree 1. Since $M$ is Gorenstein projective in $\mathrm{Mod}(A)$, each $M^{i}$ is a Gorenstein projective $A$-module. By Lemma \ref{lem 2.1}, $M$ is a Gorenstein projective $R$-module, and so is $M^{i}$ for any $i\in \mathbb{Z}$. Then there exists an exact sequence $0\rightarrow M^{i}\stackrel{f^i}\rightarrow G^{i}\rightarrow H^{i}\rightarrow 0$ in $\mathrm{Mod}(R)$ with $G^{i}$ projective and $H^{i}$ Gorenstein projective.
Let $D$ be any projective $R$-module. For any $g^i: M^i\rightarrow D$, there exists an $R$-homomorphism $h^i: G^i\rightarrow D$ such that $g^i = h^{i}f^{i}$.

Consider the following commutative diagram
$$\xymatrix@C=30pt{
\vdots\ar[d] & & &\vdots\ar[d]\\
M^{i-1}\ar[dd]_{\delta}\ar[rd]_{g^{i}\delta}\ar[rrr]^{\left(\begin{smallmatrix}f^{i-1} \\f^{i}\delta\end{smallmatrix}\right)\quad} & &
& N^{i-1}= G^{i-1}\oplus G^{i}\ar@{-->}[lld]^{(0\quad h^{i})}\ar[dd]^{\left(\begin{smallmatrix}0 & 1 \\0 & 0\end{smallmatrix}\right)}\\
&D\ar@{=}[dd] \\
M^{i}\ar[d]\ar[rd]_{g^i}\ar[rrr]^{\left(\begin{smallmatrix}f^{i} \\f^{i+1}\delta\end{smallmatrix}\right)} & &
& N^{i}= G^{i}\oplus G^{i+1}\ar@{-->}[lld]^{(h^{i}\quad 0)}\ar[d]\\
\vdots &D & &\vdots}$$
This implies that there exists an exact sequences $0\rightarrow M\rightarrow N\rightarrow L\rightarrow 0$ in $\mathrm{GrMod}(A)$ with $N$ projective, such that the induced sequence $0\rightarrow\mathrm{Hom}_{\mathrm{Gr}}(L, \overline{D}[-i])\rightarrow\mathrm{Hom}_{\mathrm{Gr}}(N, \overline{D}[-i])\rightarrow\mathrm{Hom}_{\mathrm{Gr}}(M, \overline{D}[-i])\rightarrow 0$ is still exact. Moreover, we have an exact sequence
$$0\rightarrow \mathrm{Hom}_{R}(L^{i}, D)\rightarrow \mathrm{Hom}_{R}(N^{i}, D)\rightarrow \mathrm{Hom}_{R}(M^{i}, D)\rightarrow \mathrm{Ext}_{R}^{1}(L^{i}, D)\rightarrow 0.$$ So $\mathrm{Ext}_{R}^{1}(L^{i}, D)=0$. Specifically,  $\mathrm{Ext}_{R}^{1}(L^{i}, G^{i})=0$, and then
we get the following commutative diagram:
$$\xymatrix@C=30pt{
0\ar[r] &M^{i}\ar[r]\ar@{=}[d] &N^{i}\ar[r]\ar[d]&L^{i}\ar[r]\ar[d]&0\\
0\ar[r] &M^{i}\ar[r] &G^{i}\ar[r]&H^{i}\ar[r]&0
}$$
By a version of Schanuel's Lemma, we have $L^{i}\oplus G^{i}=H^{i}\oplus N^{i}$, and then $L^{i}$ is Gorenstein projective in $\mathrm{Mod}(R)$.
So $L=\bigoplus_{i\in \mathbb{Z}}L^{i}$ is also a Gorenstein projective $R$-module.

Let $Q$ be a projective module in $\mathrm{Mod}(A)$. Then $\mathrm{Ext}_{A}^{1}(L, A\otimes_{R}Q)\cong \mathrm{Ext}_{R}^{1}(L, Q)=0$. Since $Q$ is a direct summand of $A\otimes_{R}Q$, $\mathrm{Ext}_{A}^{1}(L, Q)=0$, and then it yields from the exact sequence $0\rightarrow M\rightarrow N\rightarrow L\rightarrow 0$ in  $\mathrm{Mod}(A)$ that $L$ is a Gorenstein projective $A$-module.

Let $P\in\mathrm{GrMod}(A)$ be projective. Then $P = \prod_{i\in \mathbb{Z}}\overline{P^{i}}[-i]$ for a family of projective $R$-modules $\{P^i\}_{i\in \mathbb{Z}}$. Note that for any graded $A$-module $M$, $\mathrm{Hom}_{\mathrm{Gr}}(M, P)\cong \prod_{i\in \mathbb{Z}}\mathrm{Hom}_{R}(M^{i}, P^i)$. Then, from the exact sequence
$$0\rightarrow \prod_{i\in \mathbb{Z}}\mathrm{Hom}_{R}(L^{i}, P^i)\rightarrow \prod_{i\in \mathbb{Z}}\mathrm{Hom}_{R}(N^{i}, P^i)\rightarrow \prod_{i\in \mathbb{Z}}\mathrm{Hom}_{R}(M^{i}, P^i)\rightarrow 0,$$
we deduce the desired exact sequence $$0\longrightarrow \mathrm{Hom}_{\mathrm{Gr}}(L, P)\longrightarrow \mathrm{Hom}_{\mathrm{Gr}}(N, P)
\longrightarrow \mathrm{Hom}_{\mathrm{Gr}}(M, P)\longrightarrow 0.$$
\end{proof}

\begin{lem}\label{lem 3.5}
Let $M\in\mathrm{GrMod}(A)$. If $M$ is Gorenstein projective in $\mathrm{Mod}(A)$, then there is an exact sequence $0\rightarrow K\rightarrow N\rightarrow M\rightarrow 0$ in $\mathrm{GrMod}(A)$, where $N$ is projective and $K$ is Gorenstein projective in $\mathrm{Mod}(A)$. Moreover, it also remains exact after applying $\mathrm{Hom}_{\mathrm{Gr}}(-, P)$ for any projective module $P\in\mathrm{GrMod}(A)$.
\end{lem}

\begin{proof}
Let $M=\bigoplus_{i\in \mathbb{Z}}M^{i}\in\mathrm{GrMod}(A)$, $P$ a projective module in $\mathrm{GrMod}(A)$. Then $P = \prod_{i\in \mathbb{Z}}\overline{P^{i}}[-i]$, where $P^i$ are projective $R$-modules. Moreover, $\mathrm{Hom}_{\mathrm{Gr}}(M, P)\cong \prod_{i\in \mathbb{Z}}\mathrm{Hom}_{R}(M^{i}, P^i)$.

Since the category $\mathrm{GrMod}(A)$ has enough projectives, there exists an exact sequence $0\rightarrow K\rightarrow N\rightarrow M\rightarrow 0$ in $\mathrm{GrMod}(A)$ with $N$ projective. Considered as an exact sequence in $\mathrm{Mod}(A)$, it yields that $K$ is Gorenstein projective in $\mathrm{Mod}(A)$ since the class of Gorenstein projective modules is closed under taking kernel of epimorphisms.

Since $M^{i}$ is Gorenstein projective in $\mathrm{Mod}(A)$, it follows from Lemma \ref{lem 2.1} that $M^{i}$ is also Gorenstein projective as an $R$-module. Then the sequence
$$0\rightarrow \prod_{i\in \mathbb{Z}}\mathrm{Hom}_{R}(M^{i}, P^i)\rightarrow \prod_{i\in \mathbb{Z}}\mathrm{Hom}_{R}(N^{i}, P^i)\rightarrow \prod_{i\in \mathbb{Z}}\mathrm{Hom}_{R}(K^{i}, P^i)\rightarrow \prod_{i\in \mathbb{Z}}\mathrm{Ext}_{R}^{1}(M^{i}, P^i)=0$$ is exact. This yields the desired exact sequence $$0\longrightarrow \mathrm{Hom}_{\mathrm{Gr}}(M, P)\longrightarrow \mathrm{Hom}_{\mathrm{Gr}}(N, P)
\longrightarrow \mathrm{Hom}_{\mathrm{Gr}}(K, P)\longrightarrow 0.$$
\end{proof}

\subsection*{Proof of Theorem \ref{thm 3.1}}
(1)$\Rightarrow$(2) is precisely the result of Lemma \ref{lem 3.3}. (2)$\Rightarrow$(3) follows from Lemma \ref{lem 2.1} since $A=R[x]/(x^{2})$ is a Frobenius extension of $R$.

(2)$\Rightarrow$(1). Let $M\in \mathrm{GrMod}(A)$, and $M$ is Gorenstein projective in $\mathrm{Mod}(A)$. By Lemma \ref{lem 3.5}, there is an exact sequence $0\rightarrow K_{1}\rightarrow P_{1}\rightarrow M\rightarrow 0$ in $\mathrm{GrMod}(A)$, where $P_{1}$ is projective and $K_{1}$ is Gorenstein projective in $\mathrm{Mod}(A)$, which is also $\mathrm{Hom}_{\mathrm{Gr}}(-, P)$-exact for any projective module $P\in\mathrm{GrMod}(A)$.
Repeat this procedure, we get a $\mathrm{Hom}_{\mathrm{Gr}}(-, P)$-exact exact sequence $\cdots\rightarrow P_{2}\rightarrow P_{1}\rightarrow M\rightarrow 0$ in $\mathrm{GrMod}(A)$ with $P_{i}$ projective. Similarly, by applying Lemma \ref{lem 3.4}, we have a $\mathrm{Hom}_{\mathrm{Gr}}(-, P)$-exact exact sequence $0\rightarrow M\rightarrow P_{0}\rightarrow P_{-1}\rightarrow\cdots$ in $\mathrm{GrMod}(A)$ with $P_{i}$ projective.
Splice this two sequences together, and then we obtain a totally acyclic complex of projectives in $\mathrm{GrMod}(A)$, such that $M$ is Gorenstein projective in $\mathrm{GrMod}(A)$.

(3)$\Rightarrow$(2). By Lemma \ref{lem 2.2}(1), it suffices to construct the right part of the totally acyclic complex of projective $A$-modules. Since $M$ is a Gorenstein projective $R$-module, the argument in Lemma \ref{lem 3.4} works, that is, there is an exact sequence
$0\rightarrow M\rightarrow P_{0}\rightarrow L_{1}\rightarrow 0$ in $\mathrm{GrMod}(A)$, where $P_{0}$ is projective and $L_{1}$ is Gorenstein projective in $\mathrm{Mod}(R)$. Moreover, the sequence is $\mathrm{Hom}_{R}(-, D)$-exact for any projective $R$-module $D$. Let $P$ be any projective $A$-module. Thus, the above sequence is $\mathrm{Hom}_{A}(-, A\otimes_{R}P)$-exact, and furthermore,  $\mathrm{Hom}_{A}(-, P)$-exact.
Successively, we build a $\mathrm{Hom}_{A}(-, P)$-exact exact sequence $0\rightarrow M\rightarrow P_{0}\rightarrow P_{-1}\rightarrow\cdots$ with $P_{i}$ being projective $A$-modules. This completes the proof.
$\hfill\square$

\smallskip
Finally, let us mention recent works on $R[x]/(x^{2})$-modules. Note that $A=R[x]/(x^{2})$ is the ring of dual numbers over $R$, and differential $R$-modules (i.e. modules equipped with an $R$-endomorphism of square zero) are just $A$-modules. Avramov, Buchweitz and Iyengar \cite{ABI07} introduce projective, free and flat classes for differential modules and give some inequalities. These results specialize to basic theorems in commutative algebra and algebraic topology. Ringel and Zhang \cite{RZ17} investigate representations of quivers over the algebra of dual numbers; for a hereditary Artin algebra $R$, a bijective correspondence between the stable category of finitely generated Gorenstein projective differential $R$-modules and the category of finitely generated $R$-modules is given. Wei \cite{Wei15} shows that for any ring, a differential module is Gorenstein projective if and only if its underlying module is Gorenstein projective.

\begin{ack*}
This work was supported by National Natural Science Foundation of China (Grant No. 11401476) and China Postdoctoral Science Foundation (Grant No. 2016M591592). The author thanks Professor Chen Xiao-Wu for sharing his thoughts on this topic. This research was completed when author was a postdoctor at Fudan University supervised by Professor Wu Quan-Shui. The author thanks the referee for helpful comments and suggestions.
\end{ack*}

\end{document}